\documentclass[12pt]{amsart}
\usepackage{times,amssymb}
\usepackage{enumerate,comment}

\parskip=\smallskipamount

\newtheorem*{Thm*}{Theorem}
\newtheorem{Thm}{Theorem}

\newtheorem{Prop}[Thm]{Proposition}
\newtheorem{Lemma}[Thm]{Lemma}

\theoremstyle{definition}

\newtheorem{Remark}[Thm]{Remark}
\newtheorem{Example}[Thm]{Example}

\newcommand{\abs}[1]{\left\vert#1\right\vert}
\newcommand{\eps}{\varepsilon}
\newcommand{\set}[1]{\left\{#1\right\}}
\newcommand{\mf}[1]{\mathbb{#1}}

\DeclareMathOperator{\tr}{\mathrm{tr}}
\newcommand{\mc}[1]{\mathcal{#1}}
\newcommand{\mb}[1]{\mathbf{#1}}

\newcommand{\norm}[1]{\left \|{#1} \right \|}

\DeclareMathOperator{\E}{\mathrm{E}}
\newcommand{\Exp}[1]{\E \left[ #1 \right]}

\renewcommand{\Pr}[1]{\mathnormal{P}\left( #1 \right)}

\title{Convergence of permuted products of exponentials}
\author{Michael Anshelevich, Anh Nguyen}
\address{Department of Mathematics, Texas A\&M University, College Station, TX 77843-3368}
\email{manshel@tamu.edu, navgoingcollege2023@tamu.edu}

\begin{document}

\begin{abstract}
Let $\set{A_{i,n}}$ be a triangular array of elements in a Banach algebra, whose norms do not grow too fast, and whose row averages converge to $A$. Let $\sigma \in S(n)$ be a permutation drawn uniformly at random. If the array only contains $o(n / \log n)$ distinct elements, then almost surely, for all $0 < s < t < 1$, the permuted product of their exponentials $\prod_{i = [s n]}^{[t n]} e^{A_{\sigma(i),n}/n}$ converges in norm to $e^{(t - s) A}$. For an array of finite-dimensional matrices, convergence holds without this restriction. The proof of the latter result consists of an estimate valid in a general Banach algebra, and an application of a matrix concentration inequality.
\end{abstract}

\maketitle

\section{Introduction}

Let $\mc{M}$ be a Banach algebra, $A \in \mc{M}$, and $\set{A_{i,n} : 1 \leq i \leq n, n \in \mf{N}} \subseteq \mc{M}$ a triangular array whose row averages
\[
\frac{1}{n} \sum_{i=1}^n A_{i,n} \rightarrow A.
\]
It is easy to come up with examples, even of two-by-two matrices, for which
\[
\prod_{i=1}^n e^{A_{i,n}/n} \not \rightarrow e^A.
\]
In this note we show that, while possible, such behavior is quite unusual. Under two sets of assumptions, for most rearrangements of the array $\set{A_{i,n}: 1 \leq i \leq n}$, not only do the corresponding products converge to $e^A$, but in fact we have simultaneous approximation of $e^{t A}$ for all $t \in [0,1]$. More precisely, for each $n$, pick a permutation $\sigma \in S(n)$ uniformly at random. Then under the assumptions listed below, permuted products converge uniformly in probability
\begin{equation}
\label{Eq:Probability}
\Pr{\norm{\prod_{i=1}^{[t n]} e^{A_{\sigma(i),n}/n} - e^{t A}} < \eps \text{ for all } 0 \leq t \leq 1} \rightarrow 1 \qquad \forall \eps > 0
\end{equation}
or even uniformly almost surely
\begin{equation}
\label{Eq:AS}
\Pr{\norm{\prod_{i=1}^{[t n]} e^{A_{\sigma(i),n}/n} - e^{t A}} \rightarrow 0 \text{ for all } 0 \leq t \leq 1} = 1.
\end{equation}

The first type of assumption is when the $n$ elements in each row of the array fall into only $o(n/log n)$ distinct types.

\begin{Thm}
\label{Thm:Trotter}
In a Banach algebra $\mc{M}$, let $A \in \mc{M}$, and $\set{A_{i,n} : 1 \leq i \leq n, n \in \mf{N}} \subseteq \mc{M}$ be a uniformly bounded triangular array which satisfies $\frac{1}{n} \sum_{i=1}^n A_{i,n} \rightarrow A$.
Let $c_n = o(n/\log n)$ be a sequence of positive integers. Suppose that for each $n$, there is a family $\set{B_{j,n} : 1 \leq j \leq c_n, n \in \mf{N}}$ and an assignment $j(i)$ such that
\[
\frac{1}{n} \sum_{i=1}^n \norm{A_{i,n} - B_{j(i),n}} = o(1).
\]
Then permuted products converge almost surely in the sense of equation \eqref{Eq:AS}.
\end{Thm}

If the Banach algebra $\mc{M}$ has dimension $d < \infty$, the unit ball in $\mc{M}$ has an $\eps$-net with $\sim \eps^{-d}$ elements. It follows that for matrix algebras, the assumptions of Theorem~\ref{Thm:Trotter} hold for any uniformly bounded array, and the corresponding permuted products converge uniformly almost surely. One can allow the bound on the array to grow with $n$, at the cost of lower growth rate for $c_n$. But in the finite-dimensional case, a better growth rate is achieved by a different approach,

\begin{Thm}
\label{Thm:Main}

Let $A$ and $\set{A_{i,n}}$ be as above in the algebra of $d \times d$ complex matrices (with the operator norm). Denote
\[
L^\infty_n = \max_{1 \leq i \leq n} \norm{A_{i,n}}, \quad L^1_n = \frac{1}{n}\sum_{i=1}^n \norm{A_{i,n}}.
\]
If for some $\delta > 0$,
\begin{equation}
\label{Eq:1}
(L^1_n)^{4 + \delta} e^{3 L^1_n} L^\infty_n = o(n),
\end{equation}
then we have uniform convergence \eqref{Eq:Probability} of permuted products in probability. If in addition,
\begin{equation}
\label{Eq:2}
(L^1_n)^3 e^{3 L^1_n} L^\infty_n = o(n/\log n).
\end{equation}
the uniform convergence \eqref{Eq:AS} is almost surely.
\end{Thm}

We next list several examples where the assumptions (in the matrix case) take a more concrete form.

\begin{Example}

If $L^\infty_n = o(n)$ and $(L^1_n)$ are uniformly bounded, we have convergence of the permuted products in probability. Recall that these are precisely the standard assumptions on a triangular array of scalars with
\[
\sum_{i=1}^n \alpha_{i,n}/n \rightarrow \alpha
\]
to satisfy
\[
\prod_{i=1}^n \left( 1 + \frac{\alpha_{i,n}}{n} \right) \rightarrow e^{\alpha}.
\]
If $L^\infty_n = o(n / \log n)$, the convergence is a.s.
\end{Example}

The following is a more general example.

\begin{Example}
Take $k_n$ matrices of the same norm $L^\infty_n = o(n)$, with the remaining $n - k_n$ matrices uniformly bounded in $n$. Then for
\[
k_n = \frac{n}{3 L^\infty_n} (\log \frac{n}{L^\infty_n} - (4 + 2 \delta) \log \log \frac{n}{L^\infty_n}),
\]
we have $L^1_n \approx \frac{k_n}{n} L^\infty_n \approx \frac{1}{3} \log \frac{n}{L^\infty_n}$ unbounded, but $(L^1)^{4+\delta} e^{3 L_1} L^\infty_n = o(n)$, so the permuted products still converge in probability.

If we take
\[
k_n = \frac{n}{3 L^\infty_n} (\log \frac{n}{L^\infty_n} - \log \log n - (3 + \delta) \log \log \frac{n}{L^\infty_n}),
\]
then $(L^1)^{3} e^{3 L_1} L^\infty_n = o(\frac{n}{\log n})$, and so the permuted products converge almost surely. In particular, we have such convergence for large $L^\infty_n$,
\[
L^\infty_n = \frac{n}{\log n (\log\log n)^{3 + 2 \delta}}, \quad k_n \approx \frac{1}{3} \log n (\log\log n)^{3 + 2 \delta}.
\]
In the opposite direction, we may take all matrices to be uniformly bounded by (unbounded in $n$)
\[
L^\infty_n = \frac{1}{3} (\log n - (5 + \delta) \log \log n).
\]
For a final example of an intermediate regime, for $0 < t \leq 1$ and $0 < \alpha < 1$ or $\alpha = 1$, $\beta \leq 0$,
\[
L^\infty_n = \frac{1}{3 t} n^{1 - \alpha} (\log n)^{1 - \beta}, \quad k_n \approx \alpha t n^\alpha (\log n)^\beta.
\]

\end{Example}

To prove Theorem~\ref{Thm:Trotter}, we first show the result when $o(n/\log n)$ distinct elements occur in the array with the same frequency, and then upgrade it to arbitrary frequencies. The proof of Theorem~\ref{Thm:Main} naturally splits into two parts. In Proposition~\ref{Prop:uniform}, we translate the conditions on multiplicative convergence into conditions on the additive convergence. This part of the result still holds in an arbitrary Banach algebra. Then we prove additive convergence in Lemma~\ref{Lemma:Random} by using a concentration inequality. This part is proven only for finite dimensional matrices. Thus we do not know if the full result holds in a general Banach algebra.

\begin{Remark}[Literature review]

There is a number of articles treating topics related to this work. The following list is not meant to be complete. The techniques are typically more sophisticated than ours.

In \cite{Emme-Hubert}, the authors consider a sequence $(A_n)$ of matrices such that $\frac{1}{n} \sum_{i=1}^n A_i \rightarrow A$ and $L_n^1$ are uniformly bounded. Then for any $t \in \mf{C}$, deterministically
\[
\prod_{i=1}^n (1 + \frac{t}{n} A_i) \rightarrow e^{t A}.
\]
Note that (although the authors do not state this) their results appear to hold in a general Banach algebra.

There has been substantial follow-up work based on this article. In \cite{Henriksen-Ward}, the authors assume the matrices to be random, with mean $A$, uniformly bounded. They obtain a bound on $\norm{\prod_{i=0}^n (1 + \frac{1}{n} A_i) - e^{A}}$ which holds with almost full probability. These results are further improved in \cite{Srivastava-Concentration} and \cite{Tropp-Ward}. This last article implies that, in our context, with probability at least $1 - \delta$, for large $n$,
\[
\norm{\prod_{i=1}^n e^{A_{\sigma(i),n}/n} - e^{A}} \leq \frac{L^\infty_n e^{\norm{A}}}{\sqrt{n}} \ \sqrt{2 e^2 \log(d/\delta)}.
\]
See also \cite{Tropp-Concentration-product}. All in all, the key aim of these papers is to find the fastest rate of convergence, while our (of course, closely related) aim is to find the fastest rate of growth for the matrices under which convergence still holds.
We also emphasize that the convergence occurs uniformly to the whole path $e^{t A}$, $t \in [0,1]$, as well as the interplay between $L^\infty_n$ and $L^1_n$.

Finally, in earlier work \cite{Ans-Pritchett} of the first author with Pritchett, we considered arrays consisting of multiple copies of finitely many elements in a Banach algebra. These results are extended in Theorem~\ref{Thm:Trotter}. For other results in a Banach algebra see \cite{Lenard} and \cite{Popa-Dumitru}.

\end{Remark}

\begin{Remark}
\label{Remark:vN}

Concentration inequalities hold in some, but not all, Banach spaces, see for example \cite{Pinelis-Optimum-bounds} and \cite{Cheng-He-Luo}. So our results can be extended from matrices to these settings.

Concentration inequalities have also been extended to the context of (finite) von Neumann algebras \cite{Junge-Zeng-Poincare}, with probability replaced by the trace of the spectral projection. We however only use probability as a tool to estimate the size of a finite set, so it is not clear if these techniques apply to our context.
\end{Remark}

The paper is organized as follows. The next section treats the results in a Banach algebra. We first estimate the difference of products of exponentials in terms of local averages. Then we prove convergence for arrays of repeated matrices. In the final section, we use matrix Bernstein inequality to prove convergence in the matrix case. We finish with an example on randomized evolution families.

\section{Banach algebra results}

\begin{Remark}[Setup]
\

Choose a sequence of positive integers $a_n \rightarrow \infty$, $a_n = o(n)$. Denote $b_n = [n/a_n]$, so that also $b_n \rightarrow \infty$, $b_n = o(n)$. Divide the set $\set{1, \ldots, a_n b_n}$ into $b_n$ groups $V_{j,n}$ of consecutive elements of size $a_n$, so that for $j = 1, 2, \ldots, b_n$,
\[
V_{j,n} = \set{(j-1) a_n + 1, \ldots, (j-1) a_n + a_n}.
\]
\end{Remark}

\begin{Prop}
\label{Prop:uniform}

Let $A$ and $\set{A_{i,n} : 1 \leq i \leq n, n \in \mf{N}}$ in a Banach algebra satisfy $\frac{1}{n} \sum_{i=1}^n A_{i,n} = A_n$ with $A_n \rightarrow A$.

Suppose that $(L^1_n)^2 e^{L^1_n} = o(b_n)$,
\begin{equation}
\label{Eq:Ineq-1}
\norm{\frac{1}{a_n} \sum_{i \in V_{j,n}} A_{i, n} - A_n} e^{L^1_n} \leq \eps, \quad 1 \leq j \leq b_n,
\end{equation}
and
\begin{equation}
\label{Eq:Ineq-2}
\abs{\frac{1}{a_n} \sum_{i \in V_{j,n}} \norm{A_{i, n}} - L^1_n} e^{L^1_n} \leq \eps, \quad 1 \leq j \leq b_n.
\end{equation}
Then
\[
\limsup_{n \rightarrow \infty} \sup_{t \in [0,1]} \norm{\prod_{i=1}^{[tn]} e^{A_{i,n}/n} - e^{t A}} \leq 2 \eps.
\]
\end{Prop}

\begin{proof}
For each $j$,
\[
\begin{split}
\norm{\prod_{i \in V_{j,n}} e^{A_{i,n}/n} - e^{A_n/b_n}}
& \leq \norm{\prod_{i \in V_{j,n}} e^{A_{i,n}/n} - (1 + \sum_{i \in V_{j,n}} A_{i,n}/n)} \\
&\quad + \norm{\sum_{i \in V_{j,n}} A_{i,n}/n - A_n/b_n} + \norm{1 + A_n/b_n - e^{A_n/b_n}}.
\end{split}
\]
Using $e^{\abs{x}} - (1 + \abs{x}) \leq \frac{1}{2} \abs{x}^2 e^{\abs{x}}$ and \eqref{Eq:Ineq-2}, the first term is bounded above by
\[
\begin{split}
\prod_{i \in V_{j,n}} e^{\norm{A_{i,n}}/n} - (1 + \sum_{i \in V_{j,n}} \norm{A_{i,n}}/n)
& \leq \frac{1}{2 n^2} e^{\sum_{i \in V_{j,n}} \norm{A_{i,n}}/n} \left(\sum_{i \in V_{j,n}} \norm{A_{i,n}} \right)^2 \\
& \leq \frac{1}{2 n^2} e^{a_n (L^1_n + \eps e^{-L^1_n})/n} \left(a_n (L^1_n + \eps e^{-L^1_n}) \right)^2 \\
& \leq \frac{1}{2 b_n^2} e^{(L^1_n + \eps)/b_n} \left(L^1_n + \eps \right)^2.
\end{split}
\]
Similarly, using \eqref{Eq:Ineq-1}, the second term is bounded by $\eps e^{-L^1_n}/b_n$. Thus
\[
\norm{\prod_{i \in V_{j,n}} e^{A_{i,n}/n} - e^{A_n/b_n}}
\leq \frac{1}{2 b_n^2} e^{(L^1_n + \eps)/b_n} \left(L^1_n + \eps \right)^2 + \frac{1}{b_n} \eps e^{-L^1_n}+ \frac{1}{2 b_n^2} \norm{A_n}^2 e^{\norm{A_n} / b_n}.
\]
Also, for $k < a_n$,
\[
\begin{split}
\norm{\prod_{i=1}^{k} e^{A_{i,n}/n} - e^{k A_n/n}}
& \leq \frac{1}{2 b_n^2} e^{(L^1_n + \eps)/b_n} \left(L^1_n + \eps \right)^2 + \frac{1}{b_n} (L^1_n + \eps + \norm{A_n}) + \frac{1}{2 b_n^2} \norm{A_n}^2 e^{\norm{A_n} / b_n},
\end{split}
\]
where in the second term we use a cruder estimate \eqref{Eq:Ineq-2} instead of \eqref{Eq:Ineq-1}.



Therefore for $k = [tn] - [t n / a_n] a_n$, noting that $\norm{A_n} \leq L^1_n$,
\[
\begin{split}
\norm{\prod_{i=1}^{[tn]} e^{A_{i,n}/n} - e^{t A_n}}
& = \norm{\left(\prod_{j=1}^{[t n /a_n]} \prod_{i \in V_{j,n}} e^{A_{i,n}/n} \right) \prod_{i = [t n /a_n] a_n + 1}^{[tn]} e^{A_{i,n}/n} - \left(e^{A_n/b_n}\right)^{[t b_n]} e^{k A_n/n}} \\
& \leq \sum_{j=1}^{[t n /a_n]} \norm{\prod_{i \in V_{j,n}} e^{A_{i,n}/n} - e^{A_n/b_n}} \prod_{k \neq j} e^{\max(\sum_{i \in V_{k,n}} \norm{A_{i,n}}/n, \norm{A_n}/b_n)} \\
&\quad + \norm{\prod_{i = [t n /a_n] a_n + 1}^{[tn]} e^{A_{i,n}/n} - e^{k A_n/n}} \prod_{k \neq j} e^{\max(\sum_{i \in V_{k,n}} \norm{A_{i,n}}/n, \norm{A_n}/b_n)} \\
& \leq \sum_{j=1}^{b_n} \left( \frac{1}{2 b_n^2} e^{(L^1_n + \eps)/b_n} \left(L^1_n + \eps \right)^2 + \frac{1}{b_n} \eps e^{-L^1_n} + \frac{1}{2 b_n^2} \norm{A_n}^2 e^{\norm{A_n} / b_n} \right) \\
&\quad \times e^{(L^1_n + \eps) (b_n - 1)/ b_n} \\
&\quad + \frac{1}{b_n} \left( e^{(L^1_n + \eps)/b_n} \left(L^1_n + \eps \right)^2 + (L^1_n + \eps + \norm{A_n}) + \norm{A_n}^2 e^{\norm{A_n} / b_n} \right) \\
&\quad \times e^{(L^1_n + \eps) (b_n - 1)/ b_n} \\
& \leq \frac{3}{2 b_n} \left( e^{\eps/b_n} \left(L^1_n + \eps \right)^2 + (L^1_n + \eps + \norm{A_n}) + \norm{A_n}^2 \right) e^{L^1_n + \eps} \\
&\quad + \left( \eps e^{-L^1_n} \right) e^{L^1_n + \eps}.
\end{split}
\]
Since also $A_n \rightarrow A$, we conclude that uniformly in $t$,
\[
\limsup_{n \rightarrow \infty} \norm{\prod_{i=1}^{[tn]} e^{A_{i,n}/n} - e^{t A}}
\leq \eps e^{\eps}. \qedhere
\]

\end{proof}

\begin{Lemma}
\label{Lemma:Trotter-equal}
Let $a_n = o(n/\log n)$ be a sequence of positive integers. In a Banach algebra $\mc{M}$, let $A$ and $\set{B_{i,n} : 1 \leq i \leq a_n, n \in \mf{N}}$ be uniformly bounded and satisfy $\frac{1}{a_n} \sum_{i=1}^{a_n} B_{i,n} \rightarrow A$. Consider the array $\set{A_{i,n} : 1 \leq i \leq n, n \in \mf{N}}$, where each element $B_{i,n}$ appears $[n/a_n]$ times, with the rest of the elements uniformly bounded. For each $n$, pick a permutation $\sigma \in S(n)$ uniformly at random. Then
\[
\Pr{\norm{\prod_{i=1}^{[t n]} e^{A_{\sigma(i),n}/n} - e^{t A}} \rightarrow 0 \text{ for all } 0 \leq t \leq 1} = 1.
\]
\end{Lemma}

\begin{proof}
To simplify notation, we will assume that all $\norm{A_{i,n}} \leq 1$. The case of finite $a_n$ was already considered in \cite{Ans-Pritchett}, so we continue to assume that $a_n \rightarrow \infty$ and denote $b_n = [n/a_n]$. Without loss of generality, we may also assume that $A_{i + k a_n, n} = B_{i, n}$ for $0 \leq k < b_n$. For every permutation $\sigma \in S(n)$, let $w$ be the word in the letters $\set{B_{i, n} : 1 \leq i \leq a_n}$ obtained from it by restriction: write the word $A_{\sigma(1), n} \ldots A_{\sigma(n), n}$, and then only keep the letters of the form $B_{j,n}$. Then each word corresponds to exactly $(b_n!)^{a_n} \frac{n!}{(a_n b_n)!}$ permutations, and the uniform distribution on permutations carries over to the uniform distribution on words. Moreover, exactly one of these permutations fixes the final $(n - a_n b_n)$ elements; denote
\[
e_t^w = \prod_{i=1}^{[tn]} e^{A_{\sigma(i),n}/n}.
\]
for that permutation. Finally, denote $\overline{w}$ the standard word, corresponding to the identity permutation.

We recall some results and notation from \cite{Ans-Pritchett}. For a word $w$, denote $w_i[j]$ the number of occurrences of $B_{i,n}$ among the first $j$ letters, and
\[
\tau(w) = \frac{1}{b_n} \left( \max_{\substack{1 \leq k, \ell \leq a_n \\ 1 \leq j \leq n}} \abs{w_k[j] - w_\ell[j]} + 1 \right).
\]
Then $w$ can be transformed into the standard word $\overline{w}$ by no more than $n^2 \tau(w)$ transpositions of neighboring letters. Each transposition changes the norm of the product of exponentials by no more than
\[
\frac{1}{n^2} \max \norm{[A_{i, n}, A_{j,n}]} e^{L^1_n} \leq \frac{1}{n^2} 2 e.
\]
Therefore
\[
\norm{e_t^{\overline{w}} - e_t^w} \leq \tau(w) 2 e.
\]
Also,
\[
\Pr{\tau(w) > \frac{p_n}{\sqrt{b_n}}} \leq 2 a_n^2 \frac{\binom{2 b_n}{b_n - p_n \sqrt{b_n} + 1}}{\binom{2 b_n}{b_n}} \sim 2 a_n^2 e^{-p_n^2}
\]
Taking $p(n) = \sqrt{3 \log n}$, we thus get
\[
\begin{split}
\Pr{\norm{e_t^{\overline{w}} - e_t^w} \geq \sqrt\frac{3 a_n \log n}{n} 2 e}
& \leq 2 \frac{a_n^2}{n^3} = o\left( \frac{1}{n (\log n)^2} \right).
\end{split}
\]
By assumption, $\frac{a_n \log n}{n} = o(1)$, while the probabilities are summable. Thus by the Borel-Cantelli lemma, $\norm{e_t^w - e_t^{\overline{w}}} \rightarrow 0$ a.s.

Next, any two permutations corresponding to a word $w$ can be obtained from each other by at most $n (n - a_n b_n) \leq n a_n$ transpositions of neighboring elements. Thus for such a permutation and the corresponding word,
\[
\norm{\prod_{i=1}^{[t n]} e^{A_{\sigma(i), n}/n} - e_t^w} \leq \frac{a_n}{n} 2 e = o(1).
\]
Finally, for the standard word, the conditions in Proposition~\ref{Prop:uniform} are satisfied with $\eps = 0$. Therefore we have a Lie-Trotter formula, with the estimate
\[
\begin{split}
\norm{e_t^{\overline{w}} - e^{t A_n}}
& \leq \frac{1}{b_n} 6 e = o(1).
\end{split}
\]
Combining the three estimates, we obtain almost sure convergence of $\prod_{i=1}^{[t n]} e^{A_{\sigma(i), n}/n}$ to $e^{t A}$.
\end{proof}

\begin{Remark}
This approach works with $a_n$ distinct types of elements of the array, since a typical number of transpositions between words is $\sim n \sqrt{n a_n}$ while the maximal number is $\sim n^2$. For general permutations, both are of order $n^2$.
\end{Remark}

\begin{proof}[Proof of Theorem~\ref{Thm:Trotter}]
It follows from the assumption that
\[
\norm{\prod_{i=1}^n e^{A_{i,n}/n} - \prod_{i=1}^n e^{B_{j(i), n}/n}} = o(1).
\]
So it suffices to consider the case when $A_{i,n} = B_{j(i), n}$. We will again assume that all $\norm{B_{j,n}}$ are bounded by $1$. Choose $\alpha_n = o(1)$ so that $\alpha_n n / c_n$ is a positive integer while $c_n/\alpha_n$ is still $o(n/ \log n)$. For each $n$, denote
\[
\beta_j = \abs{\set{i : j(i) = j}}, \quad U_n = \set{j : \beta_j \geq \alpha_n n / c_n}, \quad a_n = \sum_{j \in U_n} [\beta_j c_n/\alpha_n n].
\]
Since $\sum_{j \in U_n} \beta_j \leq n$,
\[
a_n = \sum_{j \in U_n} [\beta_j c_n/\alpha_n n] \leq c_n/\alpha_n = o(n / \log n).
\]
On the other hand, since for $j \not \in U_n$, the corresponding $B_{j,n}$ appears less than $\alpha_n n/ c_n$ times,
\[
\sum_{j \in U_n} \beta_j \geq n - c_n \frac{\alpha_n n}{c_n} = (1 - \alpha_n) n.
\]
Therefore also
\[
\begin{split}
a_n \frac{\alpha_n n}{c_n} & \geq \sum_{j \in U_n} \left[ \frac{\beta_j c_n}{\alpha_n n} \right] \frac{\alpha_n n}{c_n}
\geq \sum_{j \in U_n} \beta_j - \sum_{j \in U_n} \frac{\alpha_n n }{c_n} \\
& \geq n - c_n \frac{\alpha_n n}{c_n} - \alpha_n n = n(1 - 2 \alpha_n)
\end{split}
\]
and
\[
(1 - 2 \alpha_n) (c_n / \alpha_n) \leq a_n \leq c_n / \alpha_n.
\]
Apply Lemma~\ref{Lemma:Trotter-equal} with this $a_n$ and the array $\set{B_{j, n, k} : j \in U_n, 1 \leq k \leq [\beta_j c_n/\alpha_n n]}$, where each $B_{j, n, k} = B_{j,n}$, so that the total number of elements in the array is $a_n$. It remains to verify the assumption in the lemma:
\[
\begin{split}
\norm{\frac{1}{c_n / \alpha_n} \sum_{j \in U_n} \left[\frac{\beta_j c_n}{\alpha_n n} \right] B_{j, n} - \frac{1}{n} \sum_{i=1}^{n} A_{i,n}}
& = \norm{\frac{1}{c_n / \alpha_n} \sum_{j \in U_n} \left[\frac{\beta_j c_n}{\alpha_n n} \right] B_{j, n} - \frac{1}{n} \sum_{j =1}^{c_n} \beta_j B_{j, n}} \\
& \leq \sum_{j \in U_n} \abs{\frac{\alpha_n}{c_n} \left[\frac{\beta_j c_n}{\alpha_n n} \right] - \frac{\beta_j}{n}} + \sum_{j \not \in U_n} \frac{\alpha_n}{c_n} \\
& \leq 2 \alpha_n = o(1).
\end{split}
\]
while
\[
\begin{split}
\sum_{j \in U_n} \abs{\frac{1}{c_n / \alpha_n} - \frac{1}{a_n}} \left[\frac{\beta_j c_n}{\alpha_n n} \right]
& \leq \frac{2 \alpha_n}{a_n} \sum_{j \in U_n} \left[\frac{\beta_j c_n}{\alpha_n n} \right] = 2 \alpha_n = o(1). \qedhere
\end{split}
\]
\end{proof}

\section{Matrix algebra results}

To estimate the probability that the assumptions in Proposition~\ref{Prop:uniform} are satisfied, we use the Matrix Bernstein inequality. The following weak form is sufficient for our purposes; see \cite{Tropp-Concentration-inequalities} for a stronger version.

\begin{Prop}[Matrix Bernstein inequality]

Let $X_1, \ldots, X_k$ be independent centered random $d \times d$ matrices, with $\norm{X_i} \leq L$. Let $S_k = X_1 + \ldots + X_k$. Denote $v_k = \sum_{i=1}^k \Exp{\norm{X_i}^2}$. Then
\begin{equation}
\label{Eq:Bernstein}
\Pr{\norm{S_k} > \eps} \leq 2 d \exp \left( - \frac{\eps^2/2}{v_k + L \eps /3} \right).
\end{equation}
\end{Prop}

The following result from \cite{Hoeffding,Gross-Nesme} allows us to apply the inequality above to random matrices drawn from a finite set without replacement.

\begin{Prop}

Let $C$ be a finite set of $d \times d$ matrices with sum $0$ and norms bounded by $L$. For $k \leq \abs{C}$, let $\mb{X} = (X_1, \ldots, X_k)$ be independent, drawn from $C$ uniformly with replacement, while $\mb{Y} = (Y_1, \ldots, Y_k)$ are drawn from $C$ uniformly without replacement. Denote $S_{\mb{X}} = X_1 + \ldots + X_k$ and the same for $S_{\mb{Y}}$. Then
\[
\Pr{\norm{S_\mb{Y}} > \eps} \leq \inf_{\lambda > 0} e^{-\lambda \eps} \Exp{\tr \exp (\lambda S_{\tilde{\mb{Y}}})} \leq \inf_{\lambda > 0} e^{-\lambda \eps} \Exp{\tr \exp (\lambda S_{\tilde{\mb{X}}})} \leq 2 d \exp \left( - \frac{\eps^2/2}{v_k + L \eps /3} \right),
\]
where $\tilde{\mb{Y}}$ and $\tilde{\mb{X}}$ are the Hermitian dilations of $\mb{Y}$ and $\mb{X}$.
\end{Prop}

\begin{Lemma}
\label{Lemma:Random}
Let $\set{A_{i,n} : 1 \leq i \leq n, n \in \mf{N}} \subseteq M_d(\mf{C})$ and $\frac{1}{n} \sum_{i=1}^n A_{i,n} = A_n$. For permutations $\sigma \in S(n)$ drawn uniformly at random, for $\eps < 3 L^1_n$,
\[
\Pr{\exists j : 1 \leq j \leq b_n, \norm{\frac{1}{a_n} \sum_{i \in V_{j,n}} A_{\sigma(i), n} - A_n} > \eps}
\leq b_n 2 d \exp \left( - \frac{a_n \eps^2/12}{L^1_n L^\infty_n} \right).
\]
\end{Lemma}

\begin{proof}
For each $n$, let $\set{Y_{i, n} : 1 \leq i \leq n}$ be random matrices drawn uniformly without replacement from $\set{A_{i, n} : 1 \leq i \leq n}$. Then $(Y_{1,n}, \ldots, Y_{n,n}) = (A_{\sigma(1), n}, \ldots, A_{\sigma(n), n})$, with $\sigma \in S(n)$ drawn uniformly at random. Combining the preceding propositions,
\[
\begin{split}
& \Pr{\exists j : 1 \leq j \leq b_n, \norm{\frac{1}{a_n} \sum_{i \in V_{j,n}} A_{\sigma(i), n} - A_n} > \eps} \\
&\quad = \Pr{\exists j : 1 \leq j \leq b_n, \norm{\frac{1}{a_n} \sum_{i \in V_{j,n}} Y_{i, n} - A_n} > \eps} \\
&\quad \leq \sum_{j=1}^{b_n} \Pr{\norm{\frac{1}{a_n} \sum_{i \in V_{j,n}} Y_{i, n} - A_n} > \eps} \\
&\quad \leq b_n 2 d \exp \left( - \frac{a_n^2 \eps^2/2}{v_{a_n} + 2 L^\infty_n \eps a_n /3} \right),
\end{split}
\]
where we note that $\norm{A_n} \leq L^1_n \leq L^\infty_n$. Finally,
\[
\begin{split}
v_{a_n} & = \frac{a_n}{n} \sum_{i = 1}^n \norm{A_{i,n} - A_n}^2
\leq \frac{a_n}{n} \sum_{i = 1}^n (\norm{A_{i,n}}^2 + 2 \norm{A_{i,n}} \norm{A_n} + \norm{A_n}^2) \\
& \leq a_n (L^1_n L^\infty_n + 3 (L^1_n)^2) \leq 4 a_n L^1_n L^\infty_n
\end{split}
\]
and it remains to substitute the bound on $\eps$.
\end{proof}

\begin{proof}[Proof of Theorem~\ref{Thm:Main}]

For a fixed $\eps > 0$, denote
\[
S_n = \set{\sigma \in S(n) : \text{ conditions } \eqref{Eq:Ineq-1} \text{ and } \eqref{Eq:Ineq-2} \text{ hold for the array} \set{A_{\sigma(i), n}}}.
\]
First suppose that assumption~\eqref{Eq:1} holds. Choose $a_n, b_n$ so that
\[
(L^1_n)^2 e^{L^1_n} = o(b_n), \quad \log(b_n) (L^1_n L^\infty_n e^{2 L^1_n}) \leq (L^1_n)^{2 + \delta} L^\infty_n e^{2 L^1_n} = o(a_n).
\]
Then
\[
b_n \exp \left( - \frac{a_n \eps^2/12}{L^1_n L^\infty_n e^{2 L^1_n}} \right)
\rightarrow 0.
\]
Thus by Lemma~\ref{Lemma:Random} and its scalar version, for fixed $\eps$ and large $n$, $\Pr{S_n}$ is close to $1$. The conclusion follows from the estimate in Proposition~\ref{Prop:uniform}.

Under the assumption \eqref{Eq:2}, choose $a_n, b_n$ so that
\[
(L^1_n)^2 e^{L^1_n} = o(b_n), \quad \log(n) (L^1_n L^\infty_n e^{2 L^1_n}) = o(a_n).
\]
Then
\[
\begin{split}
\sum_n b_n \exp \left( - \frac{a_n \eps^2/12}{L^1_n L^\infty_n e^{2 L^1_n}} \right)
& \leq \sum_n n \exp \left( - \frac{a_n \eps^2/12}{L^1_n L^\infty_n e^{2 L^1_n}} \right) \\
& = \sum_n n^{1 - (\eps^2/12) (a_n/\log(n) L^1_n L^\infty_n e^{2 L^1_n})} < \infty.
\end{split}
\]
The rest of the argument is similar to the one above, with the Borel-Cantelli lemma leading to the almost sure convergence.
\end{proof}

We finish the section with an additional example. See \cite{Berger-CLT-RM} and \cite{Watkins-Oseledec} for related results.





\begin{Example}

Let $A : [0,1] \rightarrow M_d(\mf{C})$ be Riemann integrable. There exists an evolution family $\set{U(s,t) : 0 \leq s \leq t \leq 1}$ which satisfies $U(s, t) \circ U(t,r) = U(s,r)$, $U(t,t) = I$, and
\[
U(s,t) = \int_s^t U(s,r) A(r) \,dr.
\]
If $A$ is continuous,
\[
\partial_t U(s,t) = U(s,t) A(t).
\]
Moreover, this family can be interpreted as a product integral or a time-ordered exponential:
\[
U(s,t) = \prod_s^t e^{A(x)} \,dx = \lim_{n \rightarrow \infty} \prod_{i= [s n] + 1}^{[t n]} e^{A(i/n)/n}.
\]
We now note that we have almost sure convergence
\[
\prod_{i = [s n] + 1}^{[t n]} e^{A_{i,n}/n} \rightarrow e^{(t-s) \int_0^1 A(x) \,dx}
\]
under either of the following assumptions. In either case, both $L^\infty_n \leq \sup_{0 \leq x \leq 1} \norm{A(x)}$, and $L^1_n \rightarrow \int_0^1 \norm{A(x)} \,dx$ are uniformly bounded.
\begin{itemize}
\item
$A_{i,n} = A(\sigma(i)/n)$, where $\sigma \in S(n)$ is drawn uniformly at random. Equivalently, $A_{i,n} = A(T_\sigma(i/n))$, where $T_\sigma$ is the (measure-preserving) bijection of $[0,1]$ permuting subintervals in its partition into $n$ parts of equal length.
\item
$A_{i,n}$ are i.i.d. drawn from the distribution of $A$ (that is, $\Pr{A_{i,n} \in S} = \abs{\set{t : A(t) \in S}}$). Intuitively, $A_{i,n} = A(T(i/n))$, where $T$ is drawn uniformly at random from the set of all measure-preserving bijections of $[0,1]$. Such a uniform distribution exists \cite{Deneckere} although it is not very natural.
\end{itemize}

\end{Example}

\def\cprime{$'$} \def\cprime{$'$}
\providecommand{\bysame}{\leavevmode\hbox to3em{\hrulefill}\thinspace}
\providecommand{\MR}{\relax\ifhmode\unskip\space\fi MR }
\providecommand{\MRhref}[2]{%
  \href{http://www.ams.org/mathscinet-getitem?mr=#1}{#2}
}
\providecommand{\href}[2]{#2}

\end{document}